\documentclass [11pt]{amsart}
\usepackage{amsmath}
\usepackage{amssymb}
\usepackage{amscd}
\usepackage{epsfig}
\usepackage{amsxtra}
\usepackage{pifont}
\usepackage{mathabx}
\usepackage{MnSymbol}
\usepackage{accents}
\usepackage{hyperref}
\usepackage{scalerel,stackengine}
\usepackage{tikz}
\usepackage{caption}
\usepackage{graphicx}
\usetikzlibrary{arrows.meta}
\usetikzlibrary{arrows,chains,matrix,positioning,scopes}
\usetikzlibrary{decorations.pathreplacing}
\usepackage{xcolor}
\usepackage{bbm}

\allowdisplaybreaks

\usepackage{geometry}
\stackMath
\newcommand\reallywidehat[1]{%
\savestack{\tmpbox}{\stretchto{%
  \scaleto{%
    \scalerel*[\widthof{\ensuremath{#1}}]{\kern-.6pt\bigwedge\kern-.6pt}%
    {\rule[-\textheight/2]{1ex}{\textheight}}
  }{\textheight}%
}{0.5ex}}%
\stackon[1pt]{#1}{\tmpbox}%
}

\newcommand\reallywidecheck[1]{%
\savestack{\tmpbox}{\stretchto{%
  \scaleto{%
    \scalerel*[\widthof{\ensuremath{#1}}]{\kern-.6pt\bigwedge\kern-.6pt}%
    {\rule[-\textheight/2]{1ex}{\textheight}}
  }{\textheight}%
}{0.5ex}}%
\stackon[1pt]{#1}{\scalebox{-1}{\tmpbox}}%
}

\numberwithin{equation}{section}

\newcommand{\RR}{{\mathbb R}}
\newcommand{\QQ}{{\mathbb Q}}
\newcommand{\ZZ}{{\mathbb Z}}
\newcommand{\CC}{{\mathbb C}}
\newcommand{\NN}{\mathbb N}
\newcommand{\cA}{\mathcal A}

 \newtheorem{theorem}{Theorem}[section]
 \newtheorem{lemma}[theorem]{Lemma}
 \newtheorem{proposition}[theorem]{Proposition}
 \newtheorem{corollary}[theorem]{Corollary}
 
 \newtheorem{conj}[theorem]{Conjecture}
 
 \newtheorem{definition}[theorem]{Definition}
 \newtheorem{example}[theorem]{Example}
  \newtheorem{remark}[theorem]{Remark}

\begin{document}
\title{Generating functions of substitutions}

\author[A. Pouti]{Aisling Pouti}
\address{Department of Mathematics and Statistics, MacEwan University, \newline
\hspace*{\parindent}  Edmonton, Alberta, Canada}
\email{poutia@mymacewan.ca}

\author[C. Ramsey]{Christopher Ramsey}
\address{Department of Mathematics and Statistics, MacEwan University, \newline
\hspace*{\parindent}  Edmonton, Alberta, Canada}
\email{ramseyc5@macewan.ca}
\urladdr{https://sites.google.com/macewan.ca/chrisramsey/}

\author[N. Strungaru]{Nicolae Strungaru}
\address{Department of Mathematics and Statistics, MacEwan University, \newline
\hspace*{\parindent}  Edmonton, Alberta, Canada, 
and 
\newline \hspace*{\parindent} 
Institute of Mathematics ``Simon Stoilow'', 
Bucharest, Romania}
\email{strungarun@macewan.ca}
\urladdr{https://sites.google.com/macewan.ca/nicolae-strungaru/home}

\keywords{Generating functions, symbolic substitutions, aperiodic, transcendental}
\subjclass[2020]{05A15, 
37B10,  	
52C23  	
}
\begin{abstract} 
In this paper we study substitutions and some of their associated generating functions. This association takes aperiodicity to transcendence, and vice-versa.
These generating functions have a recursive structure arising from the substitution which we use to study the roots of the characteristic function of the letter a in the Fibonacci substitution.
\end{abstract}

\maketitle

\section{Introduction}

One of the challenges of mathematics is to develop objects that both have enough and not too much structure, that are far-reaching in their connections, and raise many tractable questions. Substitutions readily fit these requirements, with a far ranging and well-studied theory \cite{AA,TAO, DOP, Queff}. The aim of this paper is to expand the connection between transcendental generating functions and aperiodic substitutions. 

Power series under reasonable conditions provide a very nice environment to study aperiodic structures and vice versa to study transcendence. This is because Fatou \cite{PF} showed that a rational power series with radius of convergence 1 is either rational or transcendental over $\QQ(X)$, completely removing the possibility of being irrational and algebraic. This has been used to study number-theoretic sequences \cite{BBC,BC}, context-free languages \cite{AFG}, and walks on the slit plane \cite{BM, Rubey}. Section 2 collects the necessary background for algebraic and transcendental power series for our context, including the aforementioned theorem of Fatou, Theorem \ref{thm:fatou}.

Section 3 studies generating functions for one-sided infinite words on a finite alphabet in two ways. The first is to consider the characteristic generating function of a letter in the finite alphabet, that is, taking the generating function of the integer sequence obtained by replacing that letter with a 1 and all others with a 0 in the infinite word. The second way is to consider the position-generating function of a letter, which uses the sequence of integers that give the positions of the letter in the infinite word. These generating functions encode the aperiodicity of a letter in the infinite word as transcendence, Propositions \ref{prop:2} and \ref{prop:position}.

While much can be said about infinite words in this context, Section~\ref{sect:4} studies words that arise as the fixed points of substitutions,
which are infinite one-sided words that are fixed under the substitution rule.
For every primitive substitution, there exist infinite fixed words of a power of the substitution, that arises naturally. Moreover, the fixed word can be chosen uniquely if we fix the starting letter in the alphabet. The aperiodicity of a substitution is determined from this infinite word, but can also be more simply determined by its linear algebraic structure. Theorem \ref{thm:aperiodictrans} relates the aperiodicity of a substitution with the transcendence of its affiliated characteristic and position-generating functions. 

Substitutions can also be viewed as a tiling of the infinite half-line using a finite set of tiles and a so-called inflation rule. This geometric version is in one-to-one correspondence with the symbolic substitutions. Using the coordinates of the 
endpoints of such a tiling as coefficients, we get another generating function but now with possibly irrational coefficients. In the case of two letters, and algebraic tile lengths, Theorem \ref{thm:geometric} proves that the substitution is aperiodic if and only if this geometric generating function is transcendental over $\QQ(X)$.

The last section studies the Fibonacci substitution and what can be accomplished through the recursive structure of polynomials induced by the substitution. We show that the position generating functions can be given in terms of the characteristic generating functions, Corollary \ref{cor:fibpostochar}. The rest of the section shows that this recursive structure can give us specific knowledge about the zeroes of the characteristic generating function. In particular, the characteristic function for the first letter of the Fibonacci substitution is shown to have no zeroes on $(-0.99729758,1)$, almost its whole interval of convergence, and we conjecture that it has no zeroes.

\section{Preliminaries: power series}

\subsection{Algebraic and transcendental power series}

As usual, for a field $F$ we denote by $F[X]$ the ring of polynomials with coefficients in $F$ and by $F[[X]]$ the ring of power series with coefficients in $F$. Note here that both $F[X]$ and $F[[X]]$ are integral domains.

$F(X)$ denotes the field of fractions of $F[X]$, that is 
\[
F(X)=\left\{ \frac{P(X)}{Q(X)}: P, Q \in F[X], Q \neq 0 \right\} \,.
\]

\smallskip

Recall here that given a field extension $F \subseteq K$, an element $a \in K$ is called \emph{algebraic over $F$} if there exists a nonzero polynomial $P \in F[X]$ such that $P(a)=0$. If $a \in K$ is not algebraic, then we say that $a$ is \emph{transcendental}.

\smallskip 

Given a field $F$, the ring $F[[X]]$ is an integral domain, and hence a subring of its field of fractions 
\[
F((X)):= \left\{ \frac{A}{B} : A,B \in F[[X]], B \neq  0\right\} \,.
\]

It is well known that a power series
\[
A(X)=a_0+a_1X+a_2X^2+ \ldots +a_nX^n + \ldots \in F[[X]]
\]
is a unit if and only if $a_0 \neq 0$. This immediately implies that each non-zero $B(X) \in F[[X]]$ can be written uniquely as 
\[
B(X)=X^n C(X)
\]
where $C(X)$ is a unit in $F[[X]]$. It follows that 
\[
F((X))= \left\{ \sum_{n=-N}^\infty a_nX^n : N \in \NN, a_n \in F  \right\} \,,
\]
that is, $F((X))$ is exactly the ring of formal Laurent series with coefficients in $F$.

\smallskip

Now, with the field extension
\[
F(X) \subseteq F((X)) \,,
\]
we can talk about elements $S \in F((X))$ being algebraic or transcendental over $F(X)$. A standard common denominator procedure shows that an element 
$S \in F((X))$ is algebraic over $F(X)$ if and only if there exists some polynomials $P_0, P_1, \ldots, P_n \in F[X]$, not all zero, such that 
\[
P_0(X)+ P_1(X)S(X)+ P_2(X) \left(S(X)\right)^2+\ldots + P_n(X) \left(S(X)\right)^n=0  \,.
\]

\smallskip

Below we will often be interested in whether some $A \in \RR[[X]]\subseteq \RR((X))$ is algebraic or transcendental over $\QQ(X)$. Note here that we have 
\[
\QQ(X) \subseteq \RR(X) \subseteq \RR[[X]] \subseteq \RR((X)) \,.
\]

\subsection{Radius of convergence of power series}

Consider now a power series
\[
S(X) =a_0+a_1X+ \ldots +a_nX^n+ \ldots \in \CC[[X]]
\]
and let 
\[
R:= \frac{1}{\limsup_n \sqrt[n]{|a_n|}} \,.
\]
Then the series 
\[
a_0+a_1z+ \ldots +a_nz^n+ \ldots
\]
is absolutely convergent for all $|z| < R$ and divergent for all $|z| >R$. 
$R$ is called the \emph{radius of convergence} of $S(X)$.

On the boundary of this disc, the series can be convergent at none, some or all the points. Therefore the domain $D$ of convergence satisfies
\[
\{ z \in \CC : |z| < R \} \subseteq D \subseteq \{ z \in \CC : |z| \leq R \} \,. 
\]

Throughout the paper, we will be interested in series for which $R \geq 1$, specifically this is a requirement for a theorem of Fatou introduced in the next subsection. The following is a trivial result.

\begin{lemma}\label{lem:1} Let $a_n \in \CC$ be a sequence which only takes finitely many values. Then exactly one of the following holds:
\begin{itemize}
\item[(a)] There exists some $n,N \in \NN$ such that $a_n =0$ for all $n >N$. In this case 
\[
A(X)=\sum_{n=0}^\infty a_nX^n
\]
is a polynomial and has $R=\infty$.
\item[(b)] $a_n \neq 0$ for infinitely many $n \in \NN$. In this case
\[
A(X)=\sum_{n=0}^\infty a_nX^n
\]
has a radius of convergence $R= 1$.
\end{itemize}\qed
\end{lemma}

\subsection{A theorem of Fatou and applications}

In this section we restrict to series
\[
A(X) = \sum_{n=0}^\infty a_n X^n \in \ZZ[[X]] \,.
\]
Note here that for each $n \in \NN$ we have 
\[
\sqrt[n]{|a_n|} = 0 \mbox{ or } \sqrt[n]{|a_n|} \geq 1 \,.
\]
It  follows immediately that either $A(X)$ is a polynomial or its radius of convergence satisfies $R \leq 1$. For this reason, the conditions that the series converges inside the unit disk below force $A(X)$ to either be a polynomial or have the largest possible radius of convergence, that is $R=1$. 

Let us note here in passing the following standard result, the proof of which boils down to the comparison test against a $p$-series where $p = \deg(P)$.

\begin{lemma} Let 
\[
A(X) = \sum_{n=0}^\infty a_n X^n \in \ZZ[[X]] \,.
\]
If there exists some polynomial $P(X)$ such that $|a_n| \leq |P(n)|$ for all $n \in \NN$ then $R \geq 1$. 

In particular, the conclusion holds when $a_n$ is a bounded sequence.

\qed
\end{lemma}

We can now review the Fatou's Theorem.

\begin{theorem}[Fatou \cite{PF}]\label{thm:fatou} Let $a_n \in \ZZ$ be a sequence and let $F(X)= \sum_{n=0}^\infty a_n X^n$ be its generating function. If $F$ is convergent inside the unit disk, then exactly one of the following holds:
\begin{itemize}
    \item[(a)] There exists some $P, Q \in \QQ(X)$ 
    such that 
    \[
    F(X)=\frac{P(X)}{Q(X)} \,,
    \] 
    and all roots of $Q$ are roots of unity; or,
\item[(b)] $F$ is transcendental over $\QQ(X)$.
\end{itemize}\qed
\end{theorem}

Now, the following is a generalization of a well-known fact about generating functions of $0-1$ sequences. Fatou proved a version of this \cite{PF} which has had many variations over the years (cf. \cite{BC}). For an excellent history of the subject we refer the reader to the introduction of \cite{BellChen}. Since we could not find a good reference for our formulation of (a), we  include the proof for completeness. Note that a direct proof of (b), using a deep Theorem of Cobham is provided in \cite[Sect.~5]{JPA}.

\begin{proposition}\label{fact2} Let $a_n \in \QQ$ and let 
\[
A(X)= \sum_{n=0}^\infty a_n X^n 
\]
be the generating function of $a_n$. If $a_n$ only takes finitely many values, then:
\begin{itemize}
\item[(a)] $a_n$ is eventually periodic if and only if $A(X) \in \QQ(X)$. Moreover, in this case, there exists some $P(X) \in \QQ[X]$ and $d \in \NN$ such that 
\[
A(X)=\frac{P(X)}{1-X^d} \,.
\]
\item[(b)] $a_n$ is not eventually periodic if and only if $A(X)$ is transcendental over $\QQ(X)$.
\end{itemize}
\end{proposition}
\begin{proof}
\textbf{(a)}\newline$\Longrightarrow$ Let $a_n \in \QQ$ be eventually periodic so that $a_n=a_{n+d}$ for all $n\geqslant N$ where $n,N,d \in \NN$. Let $A(X)=\sum_{n=0}^{\infty}a_nX^n$ be the generating function of $a_n$. Then we have that
\begin{align*}
A(X)=& \underbrace{a_0+a_1X+a_2X^2+\ldots++a_{N-1}X^{N-1}}_{P_1(X)}+a_NX^N+a_{N+1}X^{N+1}+a_{N+2}X^{N+2}+\ldots+a_{N+d}X^{N+d}\\
&\quad\quad + a_{N+d+1}X^{N+d+1}+a_{N+d+2}X^{N+d+2}+\ldots+a_{N+2d}X^{N+2d}+\ldots \\ 
=& P_1(X)+X^N(a_N+a_{N+1}X+a_{N+2}X^2+\ldots+a_{N+d}X^d+
a_{N+d+1}X^{d+1}+a_{N+d+2}X^{d+2}\\
&\quad\quad +\ldots+a_{N+2d}X^{2d}+\ldots)\\
=&P_1(X)+X^N(a_N+a_{N+1}X+a_{N+2}X^2+\ldots+a_{N}X^d+
a_{N+1}X^{d+1}+a_{N+2}X^{d+2}\\
&\quad\quad +\ldots+a_{N}X^{2d}+\ldots)\\
=&P_1(X)+X^NP_2(X)(1+X^d+X^{2d}+\ldots+)=P_1(X)+\frac{X^NP_2(X)}{1-X^d}=\frac{P(X)}{1-X^d} \,,
\end{align*}
where 
$$P_2(X)=a_N+a_{N+1}X+\ldots+a_{N+d-1}X^{d-1} \,,$$ and $$P(X)=P_1(X)(1-X^d)+X^NP_2(X)\,.$$.
\\

$\Longleftarrow$
We first prove the claim when $a_n \in \NN$ and then the general case.

\emph{Case 1:} $a_n \in \NN$ for all $n$. Fix some $b \in \NN$ such that, for all $n$ we have 
\[
0 \leq a_n \leq b-1 \,.
\]

Now, since $A(X) \in \QQ(X)$, there exists $P,Q \in \QQ[X]$ such that 
\[
A(X)= \frac{P(X)}{Q(X)} \,.
\]
Since $\frac{1}{b} <1$ and the series $A$ converges on the unit disk, we get 
\[
\sum_{n=0}^\infty a_n \left(\frac{1}{b}\right)^n = \frac{P(\frac{1}{b})}{Q(\frac{1}{b})} \in \QQ \,. 
\]
Now, the left hand side is just the base $b$ representation of the number
\[
\alpha = a_0.a_1a_2a_3 \ldots a_n \ldots \qquad \,.
\]
Since $a_0,\ldots, a_n, \ldots$ are digits in base $b$, and since $\alpha \in \QQ$ we get that $a_n$ is eventually periodic. This proves the claim. 

\emph{Case 2:} The general case. Let $a_n \in \QQ$ taking only finitely many values and let $m$ be the common denominator of these values. Then, the sequence 
\[
b_n :=m a_n
\]
is a sequence of integers taking only finitely many values. Set 
\[
k = \min \{ b_n : n \in \NN \} \in \ZZ \,.
\]
and 
\[
c_n:= b_n - k \,.
\]
Then, $c_n$ is a sequence of non-negative integers, taking only finitely many values. Moreover,
\begin{align*}
C(X) &= \sum_{n=0}^\infty c_n X^n = \sum_{n=0}^\infty (ma_n-k) X^n =m A(X) - \frac{k}{1-X } \in \QQ(X) \,.
\end{align*}
Therefore, by Case 1, $c_n$ is eventually periodic, and so is $a_n$. 

\textbf{(b)} Follows immediately from (a), the theorem of Fatou, and Lemma~\ref{lem:1}.
\end{proof}

\begin{remark} Under the conditions of Prop.~\ref{fact2}, if $a_n$ is a periodic sequence of integers taking only finitely many values, it follows from the proof of Prop.~\ref{fact2} that 
there exists some $P(X) \in \ZZ[X]$ and $d \in \NN$ such that 
\[
A(X)=\frac{P(X)}{1-X^d} \,.
\]
\end{remark}

As an immediate consequence we get: 

\begin{corollary}\label{cor1}
Let $b_n \in \{ 0,1 \}$ and let \[
G(X)= \sum_{n=0}^\infty b_n X^n 
\]
be the generating function of $b_n$. Then,
\begin{itemize}
\item[(a)] $b_n$ is eventually periodic if and only if $G(X) \in \QQ(X)$. Moreover, in this case, there exists some $P(X) \in \QQ[X]$ and $d \in \NN$ such that 
\[
G(X)=\frac{P(X)}{1-X^d} \,.
\]
\item[(b)] $b_n$ is not eventually periodic if and only if $G(X)$ is transcendental over $\QQ(X)$.
\end{itemize}\qed
\end{corollary}

\section{Generating functions for one-sided words}\label{worddef}

Consider an alphabet $\cA= \{ a_1, \ldots, a_k \}$, that is, a finite set of symbols. 
A one-sided infinite word $w$ is simply
\[
w=w_0w_1w_2\ldots w_n \ldots 
\]
where $w_0, \ldots ,w_n, \ldots \in \cA$.

We will primarily be interested in words coming from substitutions but we first develop the general theory for arbitrary words.

\subsection{The characteristic function of a letter}\label{charfunk}

Consider an arbitrary infinite word 
\[
w=w_0w_1 \ldots 
\]
on a finite alphabet $\cA$. Now, to each function $g: \cA \to \CC$ we can assign a formal power series
\[
C_{g}(X)= \sum_{n=0}^\infty g(w_n) X^n \,.
\]

In the particular case when $g$ is the characteristic function of $a_j$, we will simply denote the power series by $C_{a_j}$, that is
\[
C_{a_j}(X)= \sum_{n=0}^\infty \mathbbm{1}_{a_j}(w_n) X^n \,.
\]

For all $g : \cA \to \CC$ we have 
\begin{equation}\label{eq11}
C_g(X) = \sum_{j=1}^k  g(a_j) C_{a_j}(X)  \,,
\end{equation}
that is $C_{a_1}, \ldots, C_{a_k}$ is a $\CC$-basis for the vector space of such series. For this reason, 
properties of $C_{a_1}, \ldots, C_{a_k}$ will be of particular interest.

\smallskip
Let us note here in passing that when $w$ is a Sturmian word, there is a closed form in terms of the floor function for the characteristic function of each character . For example, the characteristic function of the letter $a$ in the Fibonacci word is given by
\[
\lfloor (n+1) \tau \rfloor - \lfloor n \tau \rfloor -1 \,.
\]
Unfortunately, sequences generated by such formulas typically do not lead to closed forms of the corresponding generating function. 
\smallskip

Note also in passing that if $w$ is the fixed word of a primitive substitution tiling, then each $a_j \in \cA $ appears infinitely many times, with bounded gaps. This implies that no $C_{a_j}$ is a polynomial in this case. 

We also have 
\begin{equation}\label{eq12}
\sum_{j=1}^k C_{a_j}(X)  = \frac{1}{1-X} \,.
\end{equation}

Note that each $C_{a_j}$ has coefficients in $\{ 0,1 \}$, and hence, Corollary~\ref{cor1} yields (compare \cite[Sect.~3, Lemma~3]{AFG}):
\begin{proposition}\label{prop:2}For each $1 \leq j \leq k$, the positions of $a_j$ are not eventually periodic if and only if $C_{a_j}$ is transcendental over $\QQ(X)$. \qed 
\end{proposition}

Similarly, by Proposition \ref{fact2}, we have the following two results:

\begin{proposition}\label{prop:1}\cite[Sect.~5]{JPA} Let $g: \cA \to \QQ$ be arbitrary. The sequence $g(w_n)$ is not eventually periodic if and only if $C_g(X)$ is transcendental over $\QQ(X)$.\qed 
\end{proposition}

\smallskip
\begin{remark}
In the case of one-sided fixed words for substitution tilings, a criteria for the eventual periodicity of $w$ is provided in \cite{JJP}.

It is easy to see that the following are equivalent:
\begin{itemize}
    \item[(i)] $w$ is eventually periodic.
    \item[(ii)] The characteristic function $\mathbbm{1}_{a}(w_n)$ is eventually periodic for all $a \in \cA$.
    \item[(iii)] For all $g : \cA \to \CC$, the sequence $(g(w_n))$ is eventually periodic.
\end{itemize}

On another hand, there exists substitution tilings with the property that $\mathbbm{1}_a(w_n)$ is (eventually) periodic for one letter $a \in \cA$, but not for all letters. One such example will be provided in Example~\ref{ex4.4}.
\end{remark}

\smallskip

In the case the alphabet consists of only two letters, we can also allow $g$ to take algebraic values:

\begin{proposition}\label{prop:complex}
Let $\cA = \{a_1, a_2 \}$ and $g:\cA \rightarrow \CC$ such that $g(a_1)$ and $g(a_2)$ are algebraic over $\QQ$. The sequence $g(w_n)$ is not eventually periodic if and only if $C_g(X)$ is transcendental over $\QQ(X)$. 
\end{proposition}
\begin{proof}
Equations \eqref{eq11} and \eqref{eq12} yield
\begin{align*}
C_g(X) &=g(a_1)C_{a_1}(X)+ g(a_2)C_{a_2}(X) \\
& = g(a_1)C_{a_1}(X) +g(a_2)\left(\frac{1}{1-X}-C_{a_1}(X)\right) \\
& = (g(a_1)-g(a_2))C_{a_1}(X) + g(a_1)\frac{1}{1-X} \,.
\end{align*}
The result now follows from Proposition \ref{prop:2}.
\end{proof}

\subsection{The position function of a letter}

There are an enormous amount of ways to encode the letters of a one-sided infinite word into $\CC[[X]]$. The previous subsection outlined a very sensible encoding, and we turn to a second choice here.

Again let $\cA = \{a_1,\dots, a_k\}$ be our finite alphabet and $w$ be a one-sided infinite word in $\cA$. For $1 \leq j \leq k$, define 
\[
p_j(n)= \mbox{ position of } n^{\mbox{th}}\ a_j \mbox{ in } w,
\]
sometimes referred to as $p_{a_j}(n)$, and 
\[
P_{a_j}(X)= \sum_{n=1}^{N_j} p_j(n)X^n \,.
\]
Here $N_j \in \NN \cup \{ \infty \}$ is the number of times $a_j$ appears in $w$. 

This sequence of positions is quite natural as
\[
C_{a_j}(X) = \sum_{n=1}^{N_j} X^{p_j(n)}\,.
\]

With these position-generating functions, one can certainly define $P_g$ for any $g:\cA\rightarrow \CC$ by linear combinations of the functions above. However, in many cases $\{P_{a_1},\dots, P_{a_k}\}$ will not form a basis, as they may be linearly dependent. 

\begin{example}
Let $\cA = \{a,b,c\}$ and $w$ be the $4$-periodic word 
\[
w=abcbabcbabcb\ldots \ldots 
\]
Then, 
\begin{align*}
p_1(n) & = 4n-4\\
p_2(n) & = 2n-1 \\
p_3(n) & = 4n-2
\end{align*}
are linearly dependent.
\end{example}

\begin{remark} Note here that there is a strong connection between the summatory function of the characteristic function $\mathbbm{1}_{a_j}(w_n)$ and the position function $p_j$:

For an arbitrary word $w$ and some $1 \leq j \leq k$, denote by 
\[
S_j(n):=\mathbbm{1}_{a_j}(w_0)+\mathbbm{1}_{a_j}(w_1)+\ldots + \mathbbm{1}_{a_j}(w_n)
\]
the summatory function of $\mathbbm{1}_{a_j}(n)$. Then 
\[
S_j \circ p_j = \mbox{Id}
\]
that is $p_j$ is a left inverse of $S_j$.

It is also easy to see that 
\[
p_j \circ S_j (n) = \mbox{ position of the last } a_j \mbox{ up to position } n  \,. 
\]
 \end{remark}

We wish to examine the rationality/transcendence of these generating functions similarly as to the last subsection. Obviously, some new techniques will be required 
as a letter in a word can be not eventually periodic but still have an eventually periodic structure if viewed by its position function.

If $a_j$ appears only a finite number of times then $P_{a_j}(X)$ is a polynomial. Otherwise, $\{p_j(n)\}$ is a strictly increasing non-negative sequence of integers. This implies that the radius of convergence is 
\[
\frac{1}{R} = \limsup_{n\rightarrow \infty} \sqrt[n]{p_j(n)} \geq 1.
\]
Thus, for the radius of convergence to be 1, the criterion for the theorem of Fatou, the growth of $\{p_j(n)\}$ must be subexponential. This is still a very large class of sequences, including $\{\lfloor 2^{\sqrt n}\rfloor\}$ for example.
The most tractable position functions are those which are within a finite bound of polynomial growth.
Let us prove some preliminary results about such sequences.

\begin{lemma}
Let $q(X)$ be a non-constant polynomial and $C\geq0$ be a constant. If $\{a_n\}$ is a sequence of complex numbers such that $|q(n) - a_n| \leq C$, then 
\[
|(q(n) - q(n-1)) - (a_n - a_{n-1})| \leq 2C
\]
and $\deg(q(X) - q(X-1)) = \deg(q(X)) - 1$.\qed
\end{lemma}

\begin{proposition}\label{prop:position}
Let $w$ be a one-sided infinite word in the alphabet $\cA = \{a_1,\dots, a_k\}$. For a fixed $1\leq j\leq k$, suppose there is a polynomial $q(X)$ of degree $m$ and a constant $C\geq0$ such that
\[
|q(n) - p_j(n)| \leq C\,.
\]
If 
\[
\sum_{n=1}^\infty b_n X^n := (1-X)^m\sum_{n=1}^\infty p_j(n)X^n
\]
then
\[
b_n = \sum_{l=0}^{\min\{n,m\}-1} (-1)^l\: {\binom{n}{l}}\: p_j(n-l)
\]
and $b_n$ can only take a finite number of integer values. Moreover, $P_{a_j}(X)$ is rational (resp. transcendental) over $\QQ(X)$ if and only if $(b_n)$ is eventually periodic (resp. not eventually periodic).
\end{proposition}
\begin{proof}
First, let
\begin{align*}
\sum_{n=1}^\infty b_n X^n & := (1-X)\sum_{n=1}^\infty p_j(n)X^n \\
&= p_j(1)X + \sum_{n=2}^\infty (p_j(n) - p_j(n-1))X^n\,.
\end{align*}
Let $r(X) = q(X) - q(X - 1)$ which has degree $m-1$. As well, let $D = \max\{2C, |r(0) - p_j(0)|\}$. Then by the previous lemma we have that for $n\geq 0$
\[
|r(n) - b_n| \leq D.
\]
Simply repeat this argument $m$ times to reach the first conclusion since the $\{b_n\}$ sequence of integers will be a constant away from a degree 0 polynomial, that is a constant. The last conclusion then follows from Proposition \ref{fact2}.
\end{proof}

The opposite action of taking differences is of course taking summatory series. Namely, if 
\[
\sum_{n=1}^\infty b_n X^n = (1-X)\sum_{n=1}^\infty a_n X^n \ \ \ \textrm{then} \  \ \
\frac{1}{1-X}\sum_{n=1}^\infty b_n X^n = \sum_{n=1}^\infty a_n X^n 
\]
That is, $a_n = \sum_{l=1}^n b_l$. It is of note that the summatory sequence of a sequence of integers $\{a_n\}$ that takes only a finite number of values need not be within a bound of a linear polynomial. For example, consider 
$$a_{n} =\left\{
\begin{array}{cc}
     1& \mbox{ if } n=2^m \\
     0&  \mbox{otherwise }
\end{array}
\right. \,.
$$ 

However, we can expand on the previous proposition.

\begin{proposition}
Let $w$ be a one-sided infinite word in the alphabet $\cA = \{a_1,\dots, a_k\}$. For a fixed $1\leq j\leq k$, suppose that there is an $m\in \NN$ such that the sequence
\[
b_n = \sum_{l=0}^{\min\{n,m\}-1} (-1)^l\: {\binom{n}{l}}\: p_j(n-l)
\]
takes only a finite number of values. Then $P_{a_j}(X)$ is rational or transcendental if and only if the $b_n$ are eventually periodic or not, respectively.
\end{proposition}
\begin{proof}
From the previous proposition we saw that 
\[
\sum_{n=1}^\infty b_n X^n = (1-X)^m\sum_{n=1}^\infty p_j(n)X^n
\]
which gives
\[
\sum_{n=1}^\infty p_j(n)X^n = \frac{1}{(1-X)^m}\sum_{n=1}^\infty b_n X^n\,.
\]
Hence, $P_{a_j}(X)$ has radius of convergence 1 and so the $p_j(n)$ are subexponential. 
The conclusion follows immediately.
\end{proof}

As a consequence, we get:

\begin{corollary}\label{cor:finitegaps}
Let $w$ be a one-sided infinite word in the alphabet $\cA = \{a_1,\dots, a_k\}$. If for a fixed $1\leq j\leq k$, if the difference sequence $\{p_j(n) - p_j(n-1)\}$ takes on only a finite number of values, then $P_{a_j}(X)$ is rational or transcendental over $\QQ(X)$ if and only if $a_j$ is eventually periodic in $w$ or not, respectively.
\end{corollary}
\begin{proof}
The difference sequence of the positions is the same as the gaps between successive appearances of the letter $a_j$. 
\end{proof}

Let us complete this section by remarking that the difference sequence $\{p_j(n) - p_j(n-1)\}$ takes on only a finite number of values if and only if the support of the letter $a_j$ is relatively dense in $\NN$.

\section{Generating functions for substitutions}\label{sect:4}

Recall that a substitution $\sigma$ on a finite alphabet $\cA = \{a_1,\dots, a_k\}$ is a mapping $\sigma$ of $\cA$ into the set $\cA^+$ of non-empty finite words in $\cA$. Via concatenation, $\sigma$ extends naturally to a map from $\cA^+$ into $\cA^+$. We typically add simple extra assumptions which force repeated application of $\sigma$ on every letter or word to grow.

To every such $\sigma$, we can associate a $k\times k$ matrix $A_\sigma$ where the $(i,j)$ entry equals the number of $a_j$ letters in the $\sigma(a_i)$ word.

The best-behaved substitutions are those which are primitive, that is, there is an $m\in\NN$ such that $A_\sigma^m$ has only strictly positive entries. This is equivalent to requiring the existence of some power $m$ such that each letter $a_j$ appears in $\sigma^m(a_i)$ for every letter $a_i$. Given such a primitive substitution, the Perron-Frobenius Theorem states that there exists a unique real eigenvalue of $A_\sigma$, called the PF-eigenvalue $\lambda_{PF}$, which is positive, has a strictly positive eigenvector and is larger than the modulus of any other eigenvalue.

We are interested in the one-sided infinite words that are fixed points of the substitution, meaning $\sigma(w) = w$. For primitive substitutions, fixed points always exist for some power of $\sigma$ \cite[Lemma 4.3]{TAO}.

A primitive substitution $\sigma$ on $\cA$ is called aperiodic if it has a one-sided infinite fixed word that is not eventually periodic. Note that if one is not eventually periodic, then they all are. In particular, \cite[Theorem 4.6]{TAO} gives that if the PF-eigenvalue of a primitive substitution $\sigma$ is irrational then $\sigma$ is aperiodic. Do note that the converse is false, cf. Thue-Morse substitution \cite[Section 4.6]{TAO}. For substitutions with rational (and hence integer) PF-eigenvalue, aperiodicity can be established via the existence of proximal pairs (see, for example, \cite[Cor.~4.2]{TAO}). In the case of bijective substitutions, simpler criteria are provided in \cite{BLM}.  

The reader looking for a comprehensive discussion of substitutions is directed to Chapter 4 of \cite{TAO}.

\smallskip

We start by recalling that in any fixed word of a primitive substitution, all letters appear within bounded gaps, or equivalently, that $\{p_j(n) - p_j(n-1)\}$ takes on only a finite number of values. The following result is a trivial consequence of the fact that primitive substitutions are linearly repetitive \cite{Dur}.

\begin{proposition}\label{prop:bd}
Let $\sigma$ be a primitive substitution of a finite alphabet $\cA = \{a_1,\dots, a_k\}$ of at least two letters and let $w$ be a one-sided fixed point of the substitution. There exists an $N\in \NN$ such that for all $1\leq i\leq k$ the gap between successive copies of $a_i$ in $w$ is never more than $N$.\qed
\end{proposition}

Now, we can prove the following:

\begin{theorem}\label{thm:aperiodictrans}
Let $\sigma$ be a primitive substitution of a finite alphabet $\cA = \{a_1,\dots, a_k\}$ of at least two letters and let $w$ be a one-sided fixed point of the substitution. The following are equivalent:
\begin{itemize}
    \item[(i)] $\sigma$ is aperiodic.
    \item[(ii)] There exists $1\leq i\leq k$ such that $C_{a_i}(X)$ is transcendental over $\QQ(X)$.
    \item[(iii)] There exist $1\leq i<j\leq k$ such that $C_{a_i}(X)$ and $C_{a_j}(X)$ are transcendental over $\QQ(X)$.
    \item[(iv)] There exists $1\leq i\leq k$ such that $P_{a_i}(X)$ is transcendental over $\QQ(X)$.
      \item[(v)] There exist $1\leq i<j\leq k$ such that $P_{a_i}(X)$ and $P_{a_j}(X)$ are transcendental over $\QQ(X)$.
\end{itemize}
\end{theorem}
\begin{proof}

{\bf (i) $\Longleftrightarrow$ (ii)}
The only way $\sigma$ is aperiodic is if at least one of the letters in $\cA$, say $a_i$ is not eventually periodic in $w$. By Proposition \ref{prop:2} this is equivalent to $C_{a_i}$ being transcendental over $\QQ(X).$

{\bf (iii) $\Longrightarrow$ (ii)} is obvious, while {\bf (ii) $\Longrightarrow$ (iii)} follows trivially from \eqref{eq12}:
\[
\sum_{j=1}^k C_{a_j}(X)  = \frac{1}{1-X} \,,
\]
which implies that the sum of the characteristic generating functions is a rational function over $\QQ(X)$. Hence, $C_{a_i}(X)$ is transcendental over $\QQ(X)$ if and only if at least two of these generating functions are transcendental over $\QQ(X)$.

{\bf (iii) $\Longleftrightarrow$ (v)} and {\bf (ii) $\Longleftrightarrow$ (iv)}:

By Proposition~\ref{prop:bd},  for every $1\leq i\leq k$ the difference sequence of the positions $\{p_i(n) - p_i(n-1)\}$ takes on a finite number of integers. Hence, by Corollary \ref{cor:finitegaps}, $a_i$ is not eventually periodic in $w$ if and only if $P_{a_i}$ is transcendental over $\QQ(X)$. 

By Corollary~\ref{cor1} we also have $a_i$ is not eventually periodic in $w$ if and only if $C_{a_i}$ is transcendental over $\QQ(X)$. The equivalence follows.
\end{proof}

Let us now look at some examples:

\begin{example}\label{ex:fib}
The Fibonacci substitution is given by $\cA = \{a,b\}$ with
\[
\sigma \ : \ \begin{array}{l} a \mapsto ab \\ b\mapsto a \end{array}
\]
and let
\[
w \ = \ abaababaabaababaababaabaababaabaab\dots \ = \ \lim_{m\rightarrow \infty} \sigma^m(a)\,.
\]
Then
\[
A_\sigma\  = \ \left[\begin{matrix} 1&1\\1&0 \end{matrix}\right],
\]
and $\lambda_{PF} = \frac{1 + \sqrt 5}{2}$, the golden ratio, which is irrational. Therefore, $\sigma$ is primitive and aperiodic and so, by Theorem \ref{thm:aperiodictrans}, all of $C_a(X), C_b(X), P_a(X)$, and $P_b(X)$ are all transcendental over $\QQ(X)$.
\end{example}

The transcendence of the characteristic generating functions for the Fibonacci and the Thue-Morse substitutions have been known under other guises. In particular, a more general version of this is proved in \cite{AFG} in the setting of context-free languages.  Moreover, the position functions of the Period-Doubling substitution are studied in \cite{RS}.

It is possible for $w$ to be aperiodic but for the positions of one letter to be periodic. For substitutions of constant length, the notion of height is introduced in \cite[Def.~8]{Dek} to discuss when the fixed word  $w$ is constant along an infinite arithmetic progression. An example of an aperiodic constant length substitution on $\cA=\{0,1,2\}$ for which the positions of $0$ are fully periodic is provided on \cite[Page~226]{Dek}.

We now provide a similar example, this time of a primitive substitution with an irrational inflation factor.

\begin{example}\label{ex4.4}
Consider the substitution
\[
\sigma \ : \ \begin{array}{l}x \mapsto xyzy \\
y \mapsto xy \\
z \mapsto zy \end{array}
\]
and let
\[
w \ = \ xyzyxyzyxyxyzy.... \ = \ \lim_{m\rightarrow \infty} \sigma^m(x)
\]
be its one-sided fixed point starting from $x$. Then, it is easy to see that 
\begin{align*}
w_{2k} & \in  \{x,z \} \\
w_{2k+1} & = y \,.
\end{align*}
This shows that the positions of $y$'s are periodic. On another hand, the substitution matrix is 
\[
A_\sigma = \begin{bmatrix}
 1& 2 & 1\\ 
 1& 1 & 0\\
 0 & 1 & 1
\end{bmatrix}
\]
with $\lambda_{PF} = \frac{3 + \sqrt{5}}{2}$. Hence, $\sigma$ is primitive and aperiodic. 

A fast computation shows 
\begin{align*}
C_y(X)&=\sum_{n=0}^\infty X^{2n+1} =\frac{X}{1-X^2} \in \QQ(X)\\    
P_y(X)&=\sum_{n=1}^\infty (2n-1)X^{n} =\frac{X^2+X}{(1-X)^2} \in \QQ(X) \,.    
\end{align*}
Note that by Theorem \ref{thm:aperiodictrans} and Proposition \ref{prop:2}, $C_x(X), C_z(X), P_x(X)$, and $P_z(X)$ are transcendental over $\QQ(X)$.
\end{example}

\subsection{Recursive structure}

It will come as no surprise that a substitution carries over into the structure of the $C_{a_j}(X)$ and $P_{a_j}(X)$. 

First we need more notation. For each finite word $w = w_0\dots w_l$ in the alphabet $\cA = \{a_1,\dots, a_k\}$ define, for $1\leq i\leq k$, the polynomials
\[
C_{a_i, w}(X) = \sum_{n=0}^l \mathbbm{1}_{a_i}(w_n)X^n
\]
and 
\[
P_{a_i, w}(X) = \sum_{n=1}^{\# \: a_i \: \textrm{in} \: w} p_i(n) X^n\,.
\]

Then, we have the following recursive formulas: 
\begin{proposition}\label{prop:recursive}
Let $\sigma$ be a substitution on $\cA = \{a_1,\dots, a_k\}$. For $m\in\NN$ and $1\leq i,j\leq k$, if $\sigma(a_j) = b_1\dots b_l$, then
\[
C_{a_i, \sigma^m(a_j)}(X) \ = \ \sum_{r=1}^{|\sigma(a_j)|} X^{|\sigma^{m-1}(b_1\cdots b_{r-1})|}C_{a_i, \sigma^{m-1}(b_r)}(X)
\]
and 
\[
P_{a_i, \sigma^m(a_j)}(X) \ = \ \sum_{r=1}^{|\sigma(a_j)|} X^{|\sigma^{m-1}(b_1\cdots b_{r-1})|}\left( |\sigma^{m-1}(b_1\cdots b_{r-1})|\frac{1-X^{1+\deg(P_{a_i,\sigma^{m-1}(b_r)})}}{1-X} + P_{a_i, \sigma^{m-1}(b_r)}(X)\right)
\]\qed
\end{proposition}

Let us note here in passing that all numbers $|\sigma(a_j)|, |\sigma^{m-1}(b_1\cdots b_{r-1})|$
and $\deg(P_{a_i,\sigma^{m-1}(b_r)})$ can be written explicitly in terms of entries of $A_\sigma$ and its powers. 

\smallskip

Both of the above equalities are straightforward conversions of $\sigma$ into polynomials. A careful proof of the Fibonacci case can be found in \cite{PP}, noting that they start their polynomials from the $X$ term and not the constant term.

As we have seen, any one-sided infinite fixed point $w$ of $\sigma$ will be of the form 
\[
w = \lim_{m\rightarrow \infty} \sigma^m(a_1),
\]
after possibly replacing $\sigma$ with a power and shuffling the alphabet, where convergence is in the product topology.
This implies that for $1\leq i\leq k$
\[
C_{a_i}(X) = \lim_{m\rightarrow \infty} C_{a_i, \sigma^m(a_1)}(X)
\]
in the product topology of $\QQ[[X]]$ and pointwise in $(-1,1)$. In fact, the convergence is uniform on each interval $[-r,r]$ for all $0 < r <1$.

\subsection{Geometric realisation}

Recall that substitutions give rise to one-dimensional tilings of the real line (or the half line in our one-sided case). Chapter 5 and 6 of \cite{TAO} give a very thorough development of this subject (compare \cite{AA}). Here we provide enough detail to support our results.

Let $\sigma$ be a substitution on the alphabet $\cA=\{a_1, a_2, \ldots, a_k \}$. Replace each letter $a_j$ in the alphabet by an interval $I_j$. The substitution leads to an inflation rule, or substitution tiling, if there is a common inflation factor $\lambda$ such that for each $1\leq j\leq k$, $\lambda |I_j|$ is equal to the lengths of the intervals arising from $\sigma(a_j)$.
If $\sigma$ is primitive then the PF-eigenvalue is the inflation factor and the lengths of the $I_j$ are given by the left eigenvector of $A_\sigma$ for $\lambda_{PF}$.

\begin{remark}
It is of note that every inflation factor of a primitive substitution is always algebraic, since it arises as the PF-eigenvalue of an integer matrix. If one moves to a substitution tiling using an infinite number of tiles, then transcendental inflation factors are possible \cite{FGM}.
\end{remark}

Now suppose that $w$ is a fixed one-sided infinite word for $\sigma$. Replace each letter $a_j$ in the alphabet by a translated copy of the interval $I_j$. Position these intervals to start at the origin, proceed to the right and have no gaps, and let 
\[
\Lambda =\{ 0=t_0<t_1 < \ldots < t_n < \ldots \}
\]
be the 
end points of these intervals. We will call this point set the geometric realisation of $w$.

\begin{example} Consider the symbolic substitution
\[
\sigma \ : \ \begin{array}{l} a \mapsto aba \\ b\mapsto aa \end{array} \,,
\]
and let 
\[
w=abaaaabaabaabaabaaaaba \ldots \,.
\]
Its substitution matrix is 
\[
M=\begin{bmatrix}
2 & 2 \\
1 & 0
\end{bmatrix}
\]
with eigenvalue $\lambda_{PF} =1+\sqrt{3}$. A left PF eigenvector is $(1+\sqrt{3}, 2)$. 

Therefore, we can pick two intervals $a,b$ of length $1+\sqrt{3}$ and $2$, inflate by $\lambda =1+\sqrt{3}$ and subdivide by the substitution rule:

\medskip 

\begin{tikzpicture}
\draw[->]         (4,-1)--(6,-1);
\node at (5,-.5)  {\Large$\times \lambda$};
\draw[color=red] (0,0) -- (2.73,0);
\draw[color=red] (0,-.2) -- (0,.2);
\draw[color=red] (2.73,-.2) -- (2.73,.2);
\node[color=red] at (1.36,.5)  {\Large $a$};
\draw[color=blue] (0,-2) -- (2,-2);
\draw[color=blue] (0,-2.2) -- (0,-1.8);
\draw[color=blue] (2,-2.2) -- (2,-1.8);
\node[color=blue] at (1,-1.5)  {\Large $b$};
\draw[color=red] (7,0) -- (9.73,0);
\draw[color=red] (7,-.2) -- (7,.2);
\draw[color=red] (9.73,-.2) -- (9.73,.2);
\node[color=red] at (8.36,.5)  {\Large $a$};
\draw[color=blue] (9.73,0) -- (11.73,0);
\draw[color=blue] (11.73,-0.2) -- (11.73,.2);
\node[color=blue] at (10.73,.5)  {\Large $b$};
\draw[color=red] (11.73,0) -- (14.46,0);
\draw[color=red] (14.46,-.2) -- (14.46,.2);
\node[color=red] at (13.09,.5)  {\Large $a$};
\draw[color=red] (7,-2) -- (9.73,-2);
\draw[color=red] (7,-2.2) -- (7,-1.8);
\draw[color=red] (9.73,-2.2) -- (9.73,-1.8);
\node[color=red] at (8.36,-1.5)  {\Large $a$};
\draw[color=red] (9.73,-2) -- (12.46,-2);
\draw[color=red] (12.46,-2.2) -- (12.46,-1.8);
\node[color=red] at (11.09,-1.5)  {\Large $a$};
\end{tikzpicture}
\smallskip 

Applying the geometric rule repeatedly, we obtain the following tiling of $[0, \infty)$. 

\begin{tikzpicture}[scale=.5]
\draw[color=red] (0,-.2) -- (0,.2);
\foreach \x/\y in {0/0,1/1,2/1,3/1,4/1,5/2,6/2,7/3,8/3}
{
\draw[color=red] (2.73*\x+2*\y,0) -- (2.73*\x+2*\y+2.73,0);
\draw[color=red] (2.73*\x+2*\y+2.73,-.2) -- (2.73*\x+2*\y+2.73,.2);
\node[color=red] at (2.73*\x+2*\y+1.36,.5)  {\Large $a$};
}
\foreach \x/\y in {1/0,5/1,7/2,9/3}
{
\draw[color=blue] (2.73*\x+2*\y,0) -- (2.73*\x+2*\y+2,0);
\draw[color=blue] (2.73*\x+2*\y+2,-.2) -- (2.73*\x+2*\y+2,.2);
\node[color=blue] at (2.73*\x+2*\y+1,.5)  {\Large $b$};
}
\end{tikzpicture}

\end{example}

As before we have a number of choices for generating functions.
One reasonable option is
\[
G(X) \ = \ \sum_{n=0}^\infty  t_n X^n \,.
\]

We have the following simple lemma:

\begin{lemma}\label{lem:geometric}
If $g: \cA \to \CC, \: g(a_j)=|I_j|$, then
\[
t_n = \sum_{i=0}^{n-1}g(w_i)\, , \quad \textrm{and}
\]
\[
G(X) \ = \ \frac{X}{1-X}C_g(X) \ = \ \sum_{j=1}^k \frac{X}{1-X} |I_j| C_{a_j}(X) \,.
\]
\end{lemma}
\begin{proof}
By construction, $t_n$ is the left-end point of interval $n$ (where we start with interval 0). Since $t_0=0$ this means that $t_n$, for $n\geq 1$, is equal to the sum of lengths of the first $n-1$ intervals and the formula follows.

This implies that the summatory series of $C_g(X)$ is 
\[
\frac{1}{1-X}C_g(X) = \sum_{n=0}^\infty t_{n+1}X^n.
\]
Therefore, multiplying by $X$, which shifts the sequence of coefficients, yields the desired result.
\end{proof}

As a consequence, the next proposition immediately follows from the last lemma and from Proposition \ref{prop:1}.

\begin{proposition} Let $w$ be a fixed word of a substitution and let $g: \cA \to \CC, \: g(a_j)=|I_j|$. Suppose all interval lengths are rational. Then the interval length sequence $\{g(w_n)\}$ is not eventually periodic if and only if $G(X)$ is transcendental over $\QQ(X)$. \qed
\end{proposition}

\smallskip

The next result is important for substitution tilings on two letters, as in this case the lengths of the intervals can be chosen to be algebraic numbers.

\begin{theorem}\label{thm:geometric} If the alphabet $\cA$ has two letters and the tile lengths are algebraic numbers, then exactly one of the following holds.
\begin{itemize}
\item[(a)] $|I_1|=|I_2|$. In this case 
\[
G(X)= \frac{|I_1|X}{(1-X)^2} \,.
\]
\item[(b)] $|I_1| \neq |I_2|$ but $w$ is eventually periodic. Then, there exists some $P \in \QQ[X]$ and $d \in \NN$ such that 
\[
G(X)=(|I_1|-|I_2|)\frac{XP(X)}{(1-X)(1-X^d)}+ |I_1|\frac{X}{(1-X)^2} \,.
\]
\item[(c)] $|I_1| \neq |I_2|$ but $w$ is not eventually periodic. Then $G(X)$ is transcendental over $\QQ(X)$. 
\end{itemize}
\end{theorem}
\begin{proof}
This follows directly from Proposition \ref{prop:complex}. For the extra detail in (b) we use the formula from Lemma \ref{lem:geometric}, the proof of Proposition \ref{prop:complex} and Proposition \ref{fact2}.
\end{proof}

One could also split the geometric realization generating function $G(X)$ into its constituent letters. This would combine aspects of the characteristic and position generating functions. 

\begin{remark} Given any substitution on an  alphabet $\cA=\{ a_1, \ldots, a_{k}\}$ we have the following formulas:
\begin{align*}
G(X)  &= \sum_{j=1}^k \frac{X}{1-X} |I_j| C_{a_j}(X) \\
\sum_{j=1}^k \frac{X}{1-X}  C_{a_j}(X)&= \frac{X}{(1-X)^2} \,.
\end{align*}
As we have seen in the proof of Theorem~\ref{thm:geometric}, when $k=2$ these two relations can be used to relate $G$ and $C_{a_1}$.

Unfortunately, for $k \geq 3$, such a relation does not seem possible. Of course, for constant length substitutions, we still get 
\[
G(X)=\frac{|I|X}{(1-X)^2} \,
\]
but we can say nothing more in the case $k \geq 3$.
\end{remark}

\section{The Fibonacci substitution}

As we have seen, the structure of a substitution passes to its generating functions. This will be used to give us precise information about how the characteristic and position generating functions are related and information about the roots of these functions. While this can be done for any substitution we will focus our efforts on the Fibonacci substitution, Example \ref{ex:fib}. 

First, we need some extra language to deal with the structure of the Fibonacci substitution.

\begin{definition}
     Let $A_n$ and $B_n$ denote the $n$-level supertiles of the Fibonacci substitution where
\begin{align*}
    A_n&=\sigma^n(a)\\
    B_n&=\sigma^n(b)=A_{n-1}.
\end{align*}
\end{definition}

We can now give the following formula relating the position function of each letter to the summatory function of $\mathbbm{1}_a$:

\begin{proposition}\label{prop:fibcharandpos} 
For the Fibonacci substitution $\sigma$ on $\cA = \{a,b\}$ and the one-sided infinite fixed word $w = \lim_{m\rightarrow \infty} \sigma^m(a)$, we have
\[
p_a(n)=n-2+\sum_{j=0}^{n-1} \mathbbm{1}_{a}(w_j) \,.
\]
and 
\[
p_b(n)=2n-2+\sum_{j=0}^{n-1} \mathbbm{1}_a(w_j)  \,.
\]
\end{proposition}
\begin{proof}
Consider the level one supertiles $A_1 = ab$ and $B_1=a$. The substitution allows us to write
\[
abaababaabaab \ldots \ =\  w \ =\  \sigma(w) \ =\  A_1B_1A_1A_1B_1A_1B_1A_1A_1B_1A_1A_1B_1 \ldots
\]
Now, both $A_1, B_1$ start with an $a$ and contain exactly one $a$. Thus, the $n^{th}$ $a$ appears at the beginning of the $n^{th}$ level one supertile.

There are $n-1$ supertiles before the $n^{th}$ $a$. Each supertile contains a letter, and we have to add an extra spot for each $A$ supertile. Therefore, there are $n-1+\sum_{j=1}^{n-1} \mathbbm{1}_a(w_j)$ positions before the $n^{th}$ $a$. However, the positions start at 0, so it follows that 
\[
p_a(n)= n-1+\sum_{j=0}^{n-1} \mathbbm{1}_a(w_j) - 1 \,.
\]

Next, consider the level two supertiles
\[
A_2=aba, \ B_2=ab
\]
Each of these has exactly one $b$ which appears in the second position.
Hence, the $n^{th}$ $b$ appears on the second position of the $n^{th}$ level two supertile.

There are $n-1$ supertiles before the $n^{th}$ $b$. Each supertile contains two letters, and we have to add an extra spot for each $A$ supertile. Therefore, there are $2(n-1)+\sum_{j=0}^{n-1} \mathbbm{1}_{a}(w_j)$ positions before the $n^{th}$ $b$. Again, since we start at position 0 the b is in position 1, it follows that 
\[
p_b(n)= 2(n-1)+\sum_{j=0}^{n-1} \mathbbm{1}_{a}(w_j) \,.\qedhere
\]
\end{proof}

\begin{remark} Let $S_a(n)=\sum_{j=0}^{n-1} \mathbbm{1}_{a}(w_j)$. Then, for all $n$ we have: 
\begin{align*}
p_a(n)&=n-2+S_a(n-1) \\
S_a(p_a(n))&=n \,.
\end{align*}
In particular, we have 
\[
S_a(n-2+S_a(n-1))=n \qquad \forall n \in \NN \,.
\]
    
\end{remark}

As an immediate consequence, we get the following relations between the polynomials $P_a, P_b$ and $C_a$.

\begin{corollary}\label{cor:fibpostochar}
\begin{align*}
P_a(X)& \ = \ \frac{X}{(1-X)^2} - \frac{2X}{1-X} + \frac{X}{1-X}C_a(X) \\
P_b(X)& \ =\  \frac{2X}{(1-X)^2} - \frac{2X}{1-X} + \frac{X}{1-X}C_a(X)
\end{align*}
In particular, $P_b(X)-P_a(X)= X/(1-X)^2$ is rational.
\end{corollary}



Before moving to the study of roots of $C_a(X)$, let us briefly discuss the generating function of the geometric realisation.  Starting from the one-sided fixed word $abaababa\ldots$, replacing each $a$ letter by an interval of length $\tau=\frac{1+\sqrt{5}}{2}$ and each $b$ letter by an interval of length $1$, and picking the 
endpoints of these intervals we end up with a sequence
\[
\Lambda = \{ 0 < \tau < \tau+1 <2 \tau+1 < 3 \tau +1 < \ldots \}
\]
representing the geometric realisation of the one-sided Fibonacci word. This can also be described as the Fibonacci model set intersected with the positive real axis.

A picture of the first few points in this set is drawn below, as the 
endpoints of the intervals: 

\begin{tikzpicture}[scale=.96]
\draw[color=red] (0,-.2) -- (0,.2);
\foreach \x/\y in {0/0,1/1,2/1,3/2,4/3,5/3,6/4,7/4}
{
\draw[color=red] (1.61*\x+\y,0) --(1.61*\x+\y+1.61,0);
\draw[color=red] (1.61*\x+\y+1.61,-.2) -- (1.61*\x+\y+1.61,.2);
\node[color=red] at (1.61*\x+\y+.8,.5)  {\Large $a$};
}
\foreach \x/\y in {1/0,3/1,4/2,6/3}
{
\draw[color=blue] (1.61*\x+\y,0) -- (1.61*\x+\y+1,0);
\draw[color=blue] (1.61*\x+\y+1,-.2) -- (1.61*\x+\y+1,.2);
\node[color=blue] at (1.61*\x+\y+.5,.5)  {\Large $b$};
}
\end{tikzpicture}

The generating function of the geometric realisation is then
\[
G(X)= \tau X + (\tau+1)X^2+ (2\tau+1)X^3+ (3 \tau+1)X^4+ \ldots \,.
\]
Now, \eqref{eq12} and Lemma~\ref{lem:geometric} give
\begin{align*}
G(X) &= \frac{X}{1-X} \left( \tau C_{a}(X) + C_{b}(X)\right) \\
&= \frac{X}{1-X} \left( \tau C_{a}(X) + \frac{1}{1-X}-C_{a}(X)\right) \\
&= \frac{X}{(1-X)^2} + \frac{(\tau-1)X}{1-X}   C_{a}(X) \,.
\end{align*}

\subsection{The roots of the characteristic generating function.}

Now we turn to the roots of $C_a(X)$.
If $C_a(X)$ has infinitely many zeroes in $(-1,1)$, then it would follow that $C_a(X)$ is not meromorphic on this interval, and provide a new proof of its transcendence without the use of Fatou's Theorem.

Consider
\begin{align*}
A_{3}&=\sigma^3(a)=\sigma^2(ab)=\sigma(aba)=abaab\\
B_{3}&=\sigma^3(b)=\sigma^2(a)=\sigma(ab)=aba\, ,
\end{align*}
the level 3 supertiles of the Fibonacci substitution.
Note that both of these have odd lengths, with $|A_3|=5$ and $|B_3|=3$. The level 3 supertiles are the first level where both $A_n$ and $B_n$ have the same parity. As the odd parity of both $A_n$ and $B_n$ repeats every three levels, this fact is extended to the $3n$-level supertiles $A_{3n}$ and $B_{3n}$, which will be of importance for the remainder of the section. 

Due to the recursive nature of substitutions, $\sigma$ can be applied to $3n$-level supertiles giving 
\begin{align*}
    A_{3n+3}&=A_{3n}B_{3n}A_{3n}A_{3n}B_{3n}\\
    B_{3n+3}&=A_{3n}B_{3n}A_{3n} \,.
\end{align*}
This also follows from interpreting the substitution $\sigma$ as a \textit{fusion rule}, see for example \cite{FS}.
Now using the information stated above, consider the one-sided infinite Fibonacci word made up of $3n$-level supertiles:
\begin{align*}
    w = \lim_{m\rightarrow \infty} \sigma^m(a) = A_{3n}B_{3n}A_{3n}A_{3n}B_{3n}A_{3n}B_{3n}A_{3n}\ldots \,.
\end{align*}
Split the infinite word into pairs of two consecutive $3n$-level supertiles, noting that $B_{3n}B_{3n}$ can not appear as $bb$ never occurs in $w$. The following definition states possible pairings.
\begin{definition}\label{def:triplefib} 
Let the following represent the pairings of the $3n$-level supertiles found in the Fibonacci substitution:
\begin{align*}
    R_n=A_{3n}B_{3n}\\
    S_n=A_{3n}A_{3n}\\
    T_n=B_{3n}A_{3n} \,.
\end{align*}
Then their lengths are 
\begin{align*}
    |R_n|&=f_{3n+2}+f_{3n+1}=f_{3n+3}\\
    |S_n|&=f_{3n+2}+f_{3n+2}=2f_{3n+2}\\
    |T_n|&=f_{3n+1}+f_{3n+2}=f_{3n+3} \,.
\end{align*}
\end{definition}
Note that every third number in the Fibonacci sequence is even. Thus $|R_n|$, $|S_n|$, and $|T_n|$ are even. 

Similar to Proposition \ref{prop:recursive} for finite words $w_1,w_2$ in $\cA = \{a,b\}$ and $f \in \cA$, we have
\begin{align*}
C_{f, w_1w_2}(X)=C_{f, w_1}(X)+X^{|w_1|}C_{f, w_2}(X)\,.
\end{align*}
The recursive structure of these polynomials will be heavily used. 
\begin{proposition}\label{prop:newtiling}
For each $n\geq 1$, there exist sequences of even integers $j_k,l_k,m_k$ such that
\begin{align*}
    C_{a}(X)=C_{a,{R_n}}(X)\sum_{k=0}^{\infty}X^{j_k}+C_{a,{S_n}}(X)\sum_{k=0}^{\infty}X^{l_k}+C_{a,{T_n}}(X)\sum_{k=0}^{\infty}X^{m_k}
\end{align*}
\end{proposition}
\begin{proof}
    As stated above, the one-sided infinite fixed point of the Fibonacci substitution can be described in terms of its $3n$-supertiles:
    \begin{align*}
    w & = \lim_{m\rightarrow \infty} \sigma^m(a) \\
    & = \lim_{m\rightarrow \infty} \sigma^m(A_{3n}) \\
    & = A_{3n}B_{3n}A_{3n}A_{3n}B_{3n}A_{3n}B_{3n}A_{3n}\ldots \\
    & = R_nS_nT_nT_n\ldots \qquad \,.
    \end{align*}
    Thus,
    \[
    C_a(X) \ = \ C_{a,R_n}(X) + X^{|R_n|}C_{a,S_n}(X) + X^{|R_n|+|S_n|}C_{a,T_n}(X) + X^{|R_n|+|S_n|+|T_n|}C_{a,T_n}(X)+\cdots
    \]
    Therefore, since $|R_n|, |S_n|$ and $|T_n|$ are all even, the result follows.
\end{proof}

The great advantage of this description of $C_a(X)$ is that we can get some information about its roots since we can analyze polynomials in greater detail than power series in general.

It is immediate from its definition that $C_a(X) \geq 1$ on $[0,1)$, merely because the one-sided infinite word starts with $a$.  The following are the results obtained by taking increasingly larger supertiles of $R_n$, $S_n$ and $T_n$ and using their generating functions to obtain a better understanding of the roots of $C_a(X)$ on the interval $(-1,0)$.

\begin{example}
Let $n=1$, then
\begin{align*}
R_1&\ =\ A_3B_3\ =\ abaababa\\
S_1&\ =\ A_3A_3\ =\ abaababaab\\
T_1&\ =\ B_3A_3\ =\ abaabaab
\end{align*}
which gives
\begin{align*}
    C_{a, {R_1}}(X)&\ =\ 1+X^2+X^3+X^5+X^7\\
    C_{a, {S_1}}(X)&\ =\ 1+X^2+X^3+X^5+X^7+X^8=(1+X^2+X^3)(1+X^5)\\
    C_{a, {T_1}}(X) &\ =\ 1+X^2+X^3+X^5+X^6=(1+X^2)(1+X^3)+X^6 \,.
\end{align*}
In this case, $C_{a, {S_1}}(X)>0$, $C_{a, {T_1}}(X)>0$ on $(-1,0)$ and $C_{a, {R_1}}(X)>0$ on $(-0.901593,0)$. Note that all results are verifiable via Wolfram Alpha/Maple/Mathematica. Thus, by Proposition \ref{prop:newtiling}
\[
C_a(X)>0, \ \ X\in(-0.901593,1)\,.
\]
\end{example}

\begin{example}
Let $n=2$, then
\begin{align*}
    R_2&\ =\ A_6B_6\ =\ abaababaabaababaababaabaababaabaab\\
    S_2&\ =\ A_6A_6\ =\ abaababaabaababaababaabaababaabaababaababa\\
    T_2&\ =\ B_6A_6\ =\ abaababaabaababaababaabaababaababa
\end{align*}
which gives
\begin{align*} 
C_{a,{R_2}}(X)
&=1+X^2+X^3+X^5+X^7+X^8+X^{10}+X^{11}+X^{13}+X^{15}+X^{16}+X^{18}+X^{20}+X^{21}+X^{23}\\
&\ \ +X^{24}+X^{26}+X^{28}+X^{29}+X^{31}+X^{32}\\
C_{a,{S_2}}(X)
&=1+X^2+X^3+X^5+X^7+X^8+X^{10}+X^{11}+X^{13}+X^{15}+X^{16}+X^{18}+X^{20}+X^{21}+X^{23}\\
&\ \ +X^{24}+X^{26}+X^{28}+X^{29}+X^{31}+X^{32}+X^{34}+X^{36}+X^{37}+X^{39}+X^{41}\\
C_{a,{T_2}}(X)
&=1+X^2+X^3+X^5+X^7+X^8+X^{10}+X^{11}+X^{13}+X^{15}+X^{16}+X^{18}+X^{20}+X^{21}+X^{23}\\
&\ \ +X^{24}+X^{26}+X^{28}+X^{29}+X^{31}+X^{33} \,.
\end{align*}
Therefore, $C_{a,{R_2}}(X)>0$ and $C_{a,{S_2}}(X)>0$ on $(-1,0)$ while $C_{a,{T_2}}(X)>0$ on $(-0.951699,0)$; giving that 
\[
C_a(X)>0, \ \ X\in(-0.951699,1)\,.
\]
\end{example}

These polynomials become hard to manage quite quickly. Hence, we should look at their recursive structure. 

\begin{proposition}\label{prop:recursiveRST}
For $n\in\NN$,
\begin{align*}
R_{n+1}&\ = \ R_nS_nT_nT_n\\
S_{n+1}&\ =\ R_nS_nT_nT_nR_n\\
T_{n+1}&\ =\ R_nS_nT_nR_n
\end{align*}
which gives
    \begin{align*}
        C_{a,{R_{n+1}}}(X)&=C_{a,{R_n}}(X)+X^{f_{3n+3}}C_{a,{S_n}}(X)+X^{f_{3n+3}+2f_{3n+2}}C_{a,{T_n}}(X)(1+X^{f_{3n+3}})\\
        C_{a,{S_{n+1}}}(X)&=C_{a,{R_n}}(X)+X^{f_{3n+3}}C_{a,{S_n}}(X)+X^{f_{3n+3}+2f_{3n+2}}C_{a,{T_n}}(X)(1+X^{f_{3n+3}})+X^{3f_{3n+3}+2{f_{3n+2}}}C_{a,{R_n}}(X)\\
        &=C_{a,{R_{n+1}}}(X)+X^{3f_{3n+3}+2f_{3n+2}}C_{a,{R_n}}(X)\\
        C_{a,{T_{n+1}}}(X)&=C_{a,{R_n}}(X)+X^{f_{3n+3}}C_{a,{S_n}}(X)+X^{f_{3n+3}+2f_{3n+2}}C_{a,{T_n}}(X)+X^{2f_{3n+4}}C_{a,{R_n}}(X) \,.
    \end{align*}
\end{proposition}
\begin{proof}
Everything follows by the following supertile calculations:
\begin{align*}
R_{n+1}&=A_{3n+3}B_{3n+3}=A_{3n}B_{3n}A_{3n}A_{3n}B_{3n}A_{3n}B_{3n}A_{3n}\\&=R_nS_nT_nT_n\\
S_{n+1}&=A_{3n+3}A_{3n+3}=A_{3n}B_{3n}A_{3n}A_{3n}B_{3n}A_{3n}B_{3n}A_{3n}A_{3n}B_{3n}\\&=R_nS_nT_nT_nR_n\\
T_{n+1}&=B_{3n+3}A_{3n+3}=A_{3n}B_{3n}A_{3n}A_{3n}B_{3n}A_{3n}A_{3n}B_{3n}\\&=R_nS_nT_nR_n \,.
\end{align*}
\end{proof}

Note that the first part of this proposition is really giving a new substitution in three letters $\{r,s,t\}$:
\begin{align*}
    r & \mapsto rstt \\
    s & \mapsto rsttr \\
    t & \mapsto rstr \,.
\end{align*}
This has the substitution matrix
\[
A \ = \ \left[\begin{matrix}
1 & 1 & 2 \\ 2 & 1 & 2 \\ 2& 1 & 1    
\end{matrix}\right]
\]
which is certainly primitive and has $\lambda_{PF} = 2 + \sqrt 5$. The fact that this substitution is obtained from $\sigma^3$, where $\sigma$ is still the Fibonacci substitution, is recorded in the fact that $2+\sqrt 5 = \left(\frac{1+\sqrt 5}{2}\right)^3$.

\begin{example}
With this recursion one can confidently calculate the behaviour of $n=3$ and $n=4$ but the polynomials become very large. We verified with Maple that
$C_{a,{S_3}}(X)>0$ and $C_{a,{T_3}}(X)>0$ on $(-1,0)$ and $C_{a,{R_3}}(X)>0$ on $(-0.99436269,0)$, giving that 
\[
C_a(X)>0, \ \ X\in(-0.951699,1)\,.
\]

Similarly, $C_{a,{R_4}}(X)>0$ and $C_{a,{S_4}}(X)>0$ on $(-1,0)$ while $C_{a,{T_4}}(X)>0$ on $(-0.99729758,0)$ giving that 
\[
C_a(X)>0, \ \ X\in(-0.99729758,1)\,.
\]
Note that the degrees of these polynomials are $609, \ 752$, and $608$ respectively, easily found by the formulas in Definition \ref{def:triplefib}. By the previous proposition, the number of terms and the degrees of these polynomials are growing approximately by $2+\sqrt 5$. Perhaps one could computationally squeeze out another round or two but certainly not too many more.
\end{example}

Let us now state what will happen when one looks at larger and larger values of $n$.

\begin{proposition}
For $n\in \NN$, let
\[
\alpha_n=inf\big\{ s\in[-1,0)\ :\ C_{a,{R_n}}(X)>0, C_{a,{S_n}}(X)>0, C_{a,{T_n}}(X)>0, \hspace{0.3cm}\forall X\in(s,0)\big\}\,.
\] 
Then $\alpha_{n+1} \leq \alpha_n$.
Moreover, $\alpha_n = \alpha_{n+1}$ if and only if one of the following two situations happens:
\begin{itemize}
    \item[(a)] $\alpha_m=-1$ for all $m\geq n$. In this case $C_a(X)>0$ on $(-1,1)$.
    \item[(b)] $\alpha_m = \alpha_n > -1$ for all $m\geq n$ and 
    \[
    C_{a,{R_m}}(\alpha_m)=C_{a,{S_m}}(\alpha_m)=C_{a,{T_m}}(\alpha_m)=0\,.
    \]
    In this case,
    \begin{align*}
        C_a(\alpha_n) = 0 \ \ \textrm{and} \ \ C_a(X)>0, \ \forall X \in (-\alpha_n,0).
    \end{align*}
\end{itemize}
\end{proposition}
\begin{proof}
By the recursive formulas for $C_{a,R_n}(X), C_{a,S_n}(X)$ and $C_{a,T_n}(X)$ given in Proposition \ref{prop:recursiveRST} we see that
\[
C_{a,R_n}(X), C_{a,S_n}(X), C_{a,T_n}(X) \geq 0 \ \  \Rightarrow \ \ C_{a,R_{n+1}}(X), C_{a,S_{n+1}}(X), C_{a,T_{n+1}}(X) \geq 0\,.
\]
If all three of the polynomials on the left are strictly positive  at $X$ then all three polynomials on the right are strictly positive at $X$. This implies that $\alpha_{n+1} \leq \alpha_n$.

In the same way, if at least one of the polynomials on the left is strictly positive at $X$ and the others are non-negative at $X$, then all three polynomials on the right are strictly positive at $X$. Hence, the only way $\alpha_{n+1} = \alpha_n > -1$ is if all three polynomials become zero at the same time,
\[
C_{a,{R_n}}(\alpha_n)=C_{a,{S_n}}(\alpha_n)=C_{a,{T_n}}(\alpha_n)=0\,.
\]
Thus, by Proposition \ref{prop:newtiling} this implies that $C_a(\alpha_n) = 0$ and $C_a(X) > 0$ on $(\alpha_n,1)$.
Lastly, by the recursive formulas we have for all $m>n$
\[
C_{a,{R_m}}(\alpha_m)=C_{a,{S_m}}(\alpha_m)=C_{a,{T_m}}(\alpha_m)=0\,.\qedhere
\]

\end{proof}

With all of this evidence we make the following conjecture:

\begin{conj}
$C_a(X) > 0$ on $(-1,1)$.
\end{conj}

\section*{Acknowledgements}
A.P. was supported by the NSERC USRA grant 2023-581566, A.P. and C.R. were supported by the NSERC Discovery grant 2019-05430, and N.S. was supported by the NSERC Discovery grants 2020-00038 and 2024-0485. The authors thank the reviewer for their helpful comments.

\end{document}